\documentclass[12pt]{amsart}
\usepackage[utf8]{inputenc}
\usepackage{amsthm, amscd, amsmath, mathrsfs, amssymb, amsthm, amsxtra, bbding, epsfig, eucal, graphicx, latexsym, url}
\usepackage[all]{xy}
\def \C {{\mathbb C}}

\def \N {{\mathbb N}}

\def \R {{\mathbb R}}

\def \d {\,{\rm d}}
\def\re{{\Re e\,}}
\def\im{{\Im m\,}}
\def\e{{\rm e}}
\def\i{{\rm i}}
\def\dm{\frac{1}{2}}

\def \sset {{\smallsetminus }}

\def\sumb{\mathop{{\sum}^{\flat}}}

\def\le{\leqslant}
\def\leq{\leqslant}
\def\ge{\geqslant}
\def\geq{\geqslant}

\theoremstyle{plain}
\newtheorem{theorem}{Theorem}

\newtheorem{lemma}{Lemma}[section]

\newtheorem{proposition}{Proposition}

\theoremstyle{remark}
\newtheorem{remark}{Remark}

\theoremstyle{definition}

\numberwithin{equation}{section}

\begin{document}

\vskip 5mm

\title[Fourier coefficients of cusp forms of half-integral weight]
{Non-vanishing and sign changes of
Hecke eigenvalues for half-integral weight cusp forms} 
\author{B. Chen and J. Wu}

\address{
Bin Chen
\\
Department of mathematics
\\
Tongji University
\\
Shanghai 200092
\\
China.
Department of mathematics
\\
Weinan Normal University
\\
Weinan 714000
\\
China.
}
\email{13tjccbb@tongji.edu.cn}

\address{%
Jie Wu
\\
School of Mathematics
\\
Shandong University
\\
Jinan, Shandong 250100
\\
China.
CNRS\\
Institut \'Elie Cartan de Lorraine\\
UMR 7502\\
54506 Van\-d\oe uvre-l\`es-Nancy\\
France}
\curraddr{%
Universit\'e de Lorraine\\
Institut \'Elie Cartan de Lorraine\\
UMR 7502\\
54506 Van\-d\oe uvre-l\`es-Nancy\\
France
}
\email{jie.wu@univ-lorraine.fr}

\date{\today}


\maketitle

{\hfill \textit{In honor of J. G. van der Corput's 125th birthday}}

\begin{abstract} 
In this paper, we consider three problems about signs of the Fourier coefficients of 
a cusp form $\mathfrak{f}$ with half-integral weight:
\begin{itemize}
\item[--]
The first negative coefficient of the sequence $\{\mathfrak{a}_{\mathfrak{f}}(tn^2)\}_{n\in \N}$,
\item[--]
The number of coefficients $\mathfrak{a}_{\mathfrak{f}}(tn^2)$ of same signs,
\item[--]
Non-vanishing of coefficients $\mathfrak{a}_{\mathfrak{f}}(tn^2)$ in short intervals and in arithmetic progressions,
\end{itemize}
where $\mathfrak{a}_{\mathfrak{f}}(n)$ is the $n$-th Fourier coefficient of $\mathfrak{f}$ and $t$ is a square-free integer
such that $\mathfrak{a}_{\mathfrak{f}}(t)\not=0$.

\end{abstract}

\addtocounter{footnote}{1}

\vskip 10mm

\section{Introduction}

Throughout we let $k\ge 1$ be an integer and assume $N\ge 4$ to be divisible by 4. 
Fix a real Dirichlet character $\chi$ modulo $N$. 
We write $\mathcal{S}_{k+1/2}(N, \chi)$ for the space of cusp forms of weight $k+1/2$ 
for the group $\Gamma_0(N)$ with character $\chi$. 
The space $\mathcal{S}_{3/2}(N,\chi)$  contains unary theta functions. 
Let $\mathcal{S}_{3/2}^*(N,\chi)$ be the orthogonal complement with respect to the Petersson scalar product of the subspace generated by these theta functions (\cite[Section 4]{Shimura1973} and \cite[Section 4]{Cipra1983}). 
For convenience, we put $\mathcal{S}_{k+1/2}^*(N, \chi)= \mathcal{S}_{k+1/2}(N,\chi)$ when $k\ge 2$. 
Each ${\mathfrak{f}}\in \mathcal{S}_{k+1/2}^*(N, \chi)$ has a Fourier expansion
\begin{equation}\label{deffrakfz}
{\mathfrak{f}}(z)
= \sum_{n\ge 1} {\mathfrak a}_{\mathfrak{f}}(n) \e^{2\pi\i nz} 
\qquad(\im z>0).
\end{equation}
Let ${\mathfrak{f}}\in \mathcal{S}_{k+1/2}^*(N, \chi_0)$ be a cusp form with trivial character $\chi_0$,
square-free level and real coefficients ${\mathfrak a}_{\mathfrak{f}}(n)$. 
Suppose that ${\mathfrak{f}}$ lies in the plus space,
that is, ${\mathfrak a}_{\mathfrak{f}}(n)=0$ when $(-1)^k n \equiv 2, 3\,({\rm mod}\,4)$,
see \cite{KohnenZagier1981, Kohnen1985}.
Bruinier and Kohnen \cite{BK08} gave the conjectures
\begin{equation}\label{ConjectureBK1}
\lim_{x\to\infty}
\frac{|\{n\le x : {\mathfrak a}_{\mathfrak{f}}(n)\gtrless\,0\}|}{|\{n\le x : {\mathfrak a}_{\mathfrak{f}}(n)\not=0\}|}
= \frac{1}{2}
\end{equation}
and
\begin{equation}\label{ConjectureBK2}
\lim_{x\to\infty}
\frac{|\{|d|\le x : d \; \hbox{fundamental discriminant}, \, {\mathfrak a}_{\mathfrak{f}}(|d|)\gtrless\,0\}|}
{|\{|d|\le x : d \; \hbox{fundamental discriminant}, \, {\mathfrak a}_{\mathfrak{f}}(|d|)\not=0\}|}
= \frac{1}{2}
\end{equation}
with empirical evidence, which may be, however, out of present reach. 
Alternatively, they considered the sign changes of ${\mathfrak a}_{\mathfrak{f}}(n)$ 
when $n$ runs over specific sets of integers, such as 
$\{tn^2\}_{n\in \N}$, $\{tp^{2\nu}\}_{\nu\in \N}$ and $\{tn_t^2\}_{t\,\text{square-free}}$.
Here $t$ is a positive square-free integer such that ${\mathfrak a}_{\mathfrak{f}}(t)\not=0$, 
$p$ denotes any fixed prime
and $n_t$ is an integer determined by $t$ (cf. \cite[Theorems 2.1 and 2.2]{BK08}).
In particular, they proved that the sequence $\{\mathfrak{a}_{\mathfrak{f}}(tn^2)\}_{n\in \N}$ for a fixed square-free $t$ 
has infinitely many sign changes.
Recently, Kohnen, Lau
and Wu \cite{KohnenLauWu2013} have proved some quantitative results on the number of sign changes in
this sequence.

In this paper, we shall consider the following problems:
\begin{itemize}
\item
The first negative coefficient of the sequence $\{\mathfrak{a}_{\mathfrak{f}}(tn^2)\}_{n\in \N}$;
\item
The number of coefficients $\mathfrak{a}_{\mathfrak{f}}(tn^2)$ of same signs;
\item
Non-vanishing of coefficients $\mathfrak{a}_{\mathfrak{f}}(tn^2)$ in short intervals and in arithmetic progressions.
\end{itemize}

\vskip 1mm

Denote by $n_{\mathfrak{f}}$ the smallest integer $n$ such that
\begin{equation}\label{defnf}
{\mathfrak a}_{\mathfrak{f}}(tn^2)<0
\qquad\text{and}\qquad
(n, N/2)=1.
\end{equation}
The first aim of this paper is to give an upper bound of $n_{\mathfrak{f}}$.
By using the Shimura lift of ${\mathfrak{f}}$ \cite{Shimura1973}
and developing the method of \cite{KLSW2010, LauLiuWu2012},
we can prove the following result.
 
\begin{theorem}\label{thm1}
Let $k\ge 1$ be an integer, $N\ge 4$ an integer divisible by $4$
and $\chi$ be a real Dirichlet character modulo $N$.
Suppose that ${\mathfrak{f}}\in \mathcal{S}_{k+1/2}^*(N, \chi)$ is a Hecke eigenform and 
$t\ge 1$ is a square-free integer such that ${\mathfrak a}_{\mathfrak{f}}(t)>0$.
Assume that its Shimura lift is not of CM type.
In the above notation, we have
$$
n_{\mathfrak{f}}\ll (k^2N)^{9/20},
$$
where the implied constant is absolute.
\end{theorem}

\begin{remark}
The exponent $\tfrac{9}{20}$ can be reduced to $\tfrac{3}{8}$ if we do more numerical computation as in \cite{Matomaki2012}.
\end{remark}

In order to study the number of coefficients $\mathfrak{a}_{\mathfrak{f}}(tn^2)$ of same signs,
we introduce the counting functions:
\begin{equation}\label{defNF*pmx}
\begin{cases}
\hskip 0,5mm
{\mathcal N}_{\mathfrak{f}}^{*}(x)
:= \displaystyle \sum_{\substack{n\leq x\\ {\mathfrak a}_{\mathfrak{f}}(tn^2)\neq0}} 1,
\\\noalign{\vskip 1mm}
{\mathcal N}_{\mathfrak{f}}^{+}(x)
:= \displaystyle \sum_{\substack{n\leq x\\ {\mathfrak a}_{\mathfrak{f}}(tn^2)>0}} 1,
\\\noalign{\vskip 1mm}
{\mathcal N}_{\mathfrak{f}}^{-}(x)
:= \displaystyle \sum_{\substack{n\leq x\\ {\mathfrak a}_{\mathfrak{f}}(tn^2)<0}} 1.
\end{cases}
\end{equation}
By using the method of \cite{MatomakiRadzwill2014} based on multiplicative function theory, 
we establish the following result. 

\goodbreak

\begin{theorem}\label{thm2}
Let $k\ge 1$ be an integer, $N\ge 4$ an integer divisible by $4$
and $\chi$ be a real Dirichlet character modulo $N$.
Suppose that ${\mathfrak{f}}\in \mathcal{S}_{k+1/2}^*(N, \chi)$ is a Hecke eigenform and 
$t\ge 1$ is a square-free integer such that ${\mathfrak a}_{\mathfrak{f}}(t)\not=0$.
Assume that its Shimura lift is not of CM type.
\par
{\rm (i)}
For any $\varepsilon>0$, we have
\begin{equation}\label{NonVanishing}
{\mathcal N}_{\mathfrak{f}}^{*}(x)
= \rho_{\mathfrak{f}} \, x 
\left\{1 + O_{{\mathfrak{f}}, \varepsilon}\left((\log x)^{-1/4+\varepsilon}\right)\right\}
\end{equation}
for $x\to\infty$, 
where $\delta_{\mathfrak{f}}(n)$ is the characteristic function of the integers $n$ 
such that ${\mathfrak a}_{\mathfrak{f}}(tn^2)\neq 0$ and

\begin{equation}\label{defrho}
\rho_{\mathfrak{f}} 
:= \prod_{p}\big(1-p^{-1}\big)\sum_{\nu\geq 0} \delta_{\mathfrak{f}}(p^{\nu}) p^{-\nu}>0.
\end{equation}
\par
{\rm (ii)}
For $x\to\infty$, we have
\begin{equation}\label{LowerNpmShort}
{\mathcal N}_{\mathfrak{f}}^{\pm}(x)
= \tfrac{1}{2} \rho_{\mathfrak{f}} \, x \left\{1 + O_{\mathfrak{f}, \varepsilon}\left((\log x)^{-1/4+\varepsilon}\right)\right\}.
\end{equation}
Here the implied constants depend on $\mathfrak{f}$ and $\varepsilon$.
\end{theorem}

\begin{remark}
(i)
If $N/2$ is square-free, the assumption of a non-CM Shimura lift in Theorem \ref{thm2} will automatically hold 
and hence can be omitted.
\par
(ii)
This theorem shows that the variant of \eqref{ConjectureBK1} 
for $\{{\mathfrak a}_{\mathfrak{f}}(tn^2)\}_{n\in \N}$ holds
and improves Theorem 2 of Kohnen, Lau and Wu \cite{KohnenLauWu2013}.
\end{remark}

In order to measure the non-vanishing of ${\mathfrak a}_{\mathfrak{f}}(tn^2)$, we introduce, as in \cite{Serre1981}, 
$$
i_{\mathfrak{f}}(n) := \max\{j\geq 1 \; : \; {\mathfrak a}_{\mathfrak{f}}(t(n+i)^2) = 0 \;\, \text{for \(0<i\leq j\)}\}
$$
with the convention that \(\max \emptyset = 0\). We hope to get a non-trivial bound of type 
$$
i_{\mathfrak{f}}(n)\ll_{{\mathfrak{f}}, \theta} n^{\theta}
$$
for some \(\theta<1\) and all \(n\geq 1\). 
Clearly a stronger form of the problem is to find \(y\) as small as possible 
(as a function of \(x\), \(y=x^{\theta}\) with \(\theta<1\)) such that
$$
\#\{x<n\leq x+y \; : \; \mu(n)^2=1 \;\, \text{and} \;\, {\mathfrak a}_{\mathfrak{f}}(tn^2)\neq 0\}\gg_{{\mathfrak{f}}, \theta} y,
$$
where $\mu(n)$ is the Möbius function and the implied constant can depend on ${\mathfrak{f}}$ and $\theta$. 

We have the following result.

\begin{theorem}\label{thm3}
Let $k\ge 1$ be an integer, $N\ge 4$ an integer divisible by $4$
and $\chi$ be a real Dirichlet character modulo $N$.
Suppose that ${\mathfrak{f}}\in \mathcal{S}_{k+1/2}^*(N, \chi)$ is a Hecke eigenform and 
$t\ge 1$ is a square-free integer such that ${\mathfrak a}_{\mathfrak{f}}(t)\not=0$.
Assume that its Shimura lift is not of CM type.
\par
{\rm (i)}
For every $\varepsilon>0$, $x\geq x_0({\mathfrak{f}}, \varepsilon)$ and $y\geq x^{7/17+\varepsilon}$, we have
$$
\big|\big\{x<n\leq x+y \; : \; \mu(n)^2=1 \;\, \text{and} \;\, {\mathfrak a}_{\mathfrak{f}}(tn^2)\neq 0\big\}\big|
\gg_{{\mathfrak{f}}, \varepsilon} y. 
$$
In particular for any $\varepsilon>0$ and all $n\geq 1$, we have
$$
i_{\mathfrak{f}}(n)\ll_{{\mathfrak{f}}, \varepsilon} n^{7/17+\varepsilon}.
$$
\par
{\rm (ii)}
For every $\varepsilon>0$, $x\geq x_0({\mathfrak{f}}, \varepsilon)$, $y\geq x^{17/38+100\varepsilon}$ 
and $1\leq a\leq q\leq x^{\varepsilon}$ with $(a, q)=1$, we have
$$
\big|\big\{x<n\leq x+y \; :\; \mu(n)^2=1, \; n\equiv a\,({\rm mod}\, q) \;\, \text{and} \;\, 
{\mathfrak a}_{\mathfrak{f}}(tn^2)\neq 0\big\}\big|
\gg_{{\mathfrak{f}}, \varepsilon} y/q. 
$$
Here the implied constants depend on $\mathfrak{f}$ and $\varepsilon$.
\end{theorem}

\begin{remark}
This theorem can be proved with the help of the \({\mathscr B}\)-free number method 
as in \cite{KowalskiRobertWu2007, WuZhai2013}.
Our principal tools are some estimates for multiple exponential sums (see \cite[Proposition 5]{KowalskiRobertWu2007} and \cite[Proposition 2.1]{WuZhai2013}).
J. G. van der Corput is a pioneer on this domain. 
On the occasion of his 125th birthday,
it is a pleasure for us to write this paper to commemorate the outstanding contribution he made to the number theory.
\end{remark}

\medskip

\noindent{\bf Acknowledgement}.
The authors deeply thank the referee for valuable comments and suggestions.
This work was began when the second author visited Weinan Normal University at the begining of 2015.
He would like to thank the institute for the pleasant working conditions.
B. Chen and J. Wu are supported in part by NSFC (Grant No. 61402335) and IRT1264, respectively.

\vskip 10mm

\section{Background}

\smallskip

For ${\mathfrak{f}}\in \mathcal{S}_{k+1/2}^*(N, \chi)$, 
let $t\ge 1$ be a square-free integer such that ${\mathfrak a}_{\mathfrak{f}}(t)\not=0$.
The Shimura correspondence \cite{Shimura1973} lifts ${\mathfrak{f}}$ to a cusp form $f_t$ of weight $2k$ 
for the group $\Gamma_0(N/2)$ with character $\chi^2$. 
Also it gives a vital relation between their Fourier coefficients, 
\begin{equation}\label{defaft}
a_{f_t}(n)
:= \sum_{d\mid n} \chi_{t,N}(d) d^{k-1} 
{\mathfrak a}_{\mathfrak{f}}\bigg(t\frac{n^2}{d^2}\bigg),
\end{equation}
where $\chi_{t,N}$ denotes the character 
\begin{equation}\label{defchitN}
\chi_{t,N}(d) 
:=  \chi(d)\left(\frac{(-1)^kt}{d}\right)
\end{equation} 
($\big(\frac{(-1)^kt}{d}\big)$ is Legendre's symbole)
and  
\begin{equation}\label{deffz}
f_t(z)
:= \sum_{n\ge 1} a_{f_t}(n) \e^{2\pi\i nz} 
\qquad
(\im z>0).
\end{equation}
($f_t$ is called the Shimura lift of ${\mathfrak{f}}$ associated to $t$.)
Furthermore, if ${\mathfrak{f}}$ is a Hecke eigenform, then so is the Shimura lift of $\mathfrak{f}$. 
In fact, in this case, 
\begin{eqnarray}\label{sl}
f(z) := \mathfrak{a}_{\mathfrak{f}}(t)^{-1} f_t(z)
\end{eqnarray} 
is a normalized Hecke eigenform independent of $t$ (i.e. the first coefficient of Fourier is equal to 1),
and the arithmetic function $n \mapsto {\mathfrak a}_{\mathfrak{f}}(tn^2)$ is multiplicative 
in the following sense (cf. \cite[(1.18)]{Shimura1973}):
\begin{equation}\label{multiplicativity}
{\mathfrak a}_{\mathfrak{f}}(tm^2) {\mathfrak a}_{\mathfrak{f}}(tn^2)
= {\mathfrak a}_{\mathfrak{f}}(t) {\mathfrak a}_{\mathfrak{f}}(tm^2n^2)
\quad{\rm if}\quad
(m, n)=1.
\end{equation}

Write
\begin{equation}\label{deflambdaf_a*f}
\lambda_f(n) 
:= \mathfrak{a}_{\mathfrak{f}}(t)^{-1} a_{f_t}(n) n^{-(2k-1)/2},
\qquad
\mathfrak{a}_{\mathfrak{f}}^*(n) 
:= \mathfrak{a}_{\mathfrak{f}}(t)^{-1} \mathfrak{a}_{\mathfrak{f}}(tn^2) n^{-(k-1/2)}.
\end{equation}
Clearly $\lambda_f(n)$ is the $n$-th normalized Fourier coefficient of $f$ (i.e., $\lambda_f(1)=1$)
and the function $n\mapsto \mathfrak{a}_{\mathfrak{f}}^*(n)$ is multiplicative.

Further we introduce the following notation:
\begin{equation}\label{defLNk}
\mathcal{L} := \log(C_0kN)
\qquad
\text{and}
\qquad
N_k := \prod_{p\le \mathcal{L}^2\,\text{or}\;p\mid (N/2)} p,
\end{equation}
where $C_0$ is an absolute large constant.
Write
\begin{equation*}
\left\{
\begin{array}{rl}
L(s, f)
\hskip -2,5mm & := \displaystyle\sum_{n\ge 1} \lambda_f(n) n^{-s},
\\\noalign{\vskip 1,3mm}
L(s, \mathfrak{a}_{\mathfrak{f}}^*)
\hskip -2,5mm & := \displaystyle\sum_{n\ge 1} \mathfrak{a}_{\mathfrak{f}}^*(n) n^{-s},
\\\noalign{\vskip 1,3mm}
L(s, \chi_{t, N})
\hskip -2,5mm & :=\displaystyle \sum_{n\ge 1} \chi_{t, N}(n) n^{-s},
\end{array}\right.
\qquad
\left\{
\begin{array}{rl}
L^{\flat}(s, f)
\hskip -2,5mm & := \displaystyle\sideset{}{^\flat}\sum_{n\ge 1} \lambda_f(n) n^{-s},
\\\noalign{\vskip 0,5mm}
L^{\flat}(s, \mathfrak{a}_{\mathfrak{f}}^*)
\hskip -2,5mm & := \displaystyle\sideset{}{^\flat}\sum_{n\ge 1} \mathfrak{a}_{\mathfrak{f}}^*(n) n^{-s},
\\\noalign{\vskip 0,5mm}
L^{\flat}(s, \chi_{t, N})
\hskip -2,5mm & := \displaystyle\sideset{}{^\flat}\sum_{n\ge 1} \chi_{t, N}(n)n^{-s},
\end{array}\right.
\end{equation*}
where $\sum^{\flat}$ means that the sum runs over square-free integers $n$ satisfying $(n, N_k)=1$.

\begin{lemma}\label{lem2.1}
Let $k\ge 1$ be an integer, $N\ge 4$ an integer divisible by $4$
and $\chi$ be a real Dirichlet character modulo $N$.
Suppose that ${\mathfrak{f}}\in \mathcal{S}_{k+1/2}^*(N, \chi)$ is a Hecke eigenform and 
$t\ge 1$ is a square-free integer such that ${\mathfrak a}_{\mathfrak{f}}(t)\not=0$.
\par
{\rm (i)}
We have
\begin{equation}\label{Series_Convolution_1}
L(s, f)
= L(s+\tfrac{1}{2}, \chi_{t, N}) L(s, \mathfrak{a}_{\mathfrak{f}}^*),
\end{equation}
for all $s\in \C$; and
\begin{equation}\label{Series_Convolution_2}
L^{\flat}(s, f)
= L^{\flat}(s+\tfrac{1}{2}, \chi_{t, N}) L^{\flat}(s, \mathfrak{a}_{\mathfrak{f}}^*) \mathscr{L}_f(s),
\end{equation}
for $\sigma>\tfrac{1}{2}$,
where the Dirichlet series $\mathscr{L}_f(s)$ converges absolutely for $\sigma>\tfrac{1}{2}$ and
\begin{equation}\label{Estimate:mathscrLfs}
\mathscr{L}_f(s)\asymp_{\varepsilon} 1
\qquad
(\sigma\ge \tfrac{1}{2}+\varepsilon).
\end{equation}
Here the implied constant depends only on $\varepsilon$.
\par
{\rm (ii)}
For any $\varepsilon>0$, we have
\begin{equation}\label{LowerBound_Dirichlet_1}
L(\sigma+{\rm i}\tau, \chi_{t, N})^{-1}
\ll \varepsilon^{-1}
\end{equation}
and
\begin{equation}\label{LowerBound_Dirichlet_2}
L^{\flat}(\sigma+{\rm i}\tau, \chi_{t, N})^{-1}
\ll \varepsilon^{-2}
\end{equation}
for $\sigma\ge 1+\varepsilon$ and $\tau\in \R$,
where the implied constants are absolute.
\par
{\rm (iii)}
For any $\varepsilon>0$, we have
\begin{equation}\label{ConvexityBounds_1}
|L(\sigma+{\rm i}\tau, \mathfrak{a}_{\mathfrak{f}}^*)|
\ll_{\varepsilon} \big(N(k+|\tau|)^2\big)^{\max\{(1-\sigma)/2, \, 0\}+\varepsilon}
\end{equation}
and
\begin{equation}\label{ConvexityBounds_2}
|L^{\flat}(\sigma+{\rm i}\tau, \mathfrak{a}_{\mathfrak{f}}^*)|
\ll_{\varepsilon} \big(N(k+|\tau|)^2\big)^{\max\{(1-\sigma)/2, \, 0\}+\varepsilon}
\end{equation}
for $\sigma\ge \tfrac{1}{2}+\varepsilon$ and $\tau\in \R$,
where the implied constants depend only on $\varepsilon$.
\end{lemma}

\begin{proof}
In view of the definition \eqref{deflambdaf_a*f}, the formula \eqref{defaft} is equivalent to
\begin{equation}\label{Coefficient_Convolution}
\lambda_f(n)
= \sum_{d\mid n} \frac{\chi_{t,N}(d)}{\sqrt{d}} {\mathfrak a}_{\mathfrak{f}}^*\bigg(\frac{n}{d}\bigg).
\end{equation}
Clearly this formula implies \eqref{Series_Convolution_1} for $\sigma>1$.
By analytic continuation, this relation is true for all $s\in \C$ 
since $L(s, f)$ and $L(s+\tfrac{1}{2}, \chi_{t, N})$ are entire. 

Put $g(n):=\chi_{t,N}(n)\mu(n)^2/\sqrt{n}$ and $h(n):={\mathfrak a}_{\mathfrak{f}}^*(n) \mu(n)^2$,
where $\mu(n)$ is the Möbius function.
Define the multiplicative function $\ell(n)$ by the relation 
\begin{equation}\label{def:elln}
\lambda_f(n)\mu(n)^2 = (g*h*\ell)(n).
\end{equation}
In view of \eqref{Coefficient_Convolution}, we have $\ell(p)=0$ for all primes $p$.
Next we shall prove that there is an absolute constant $C$ such that
\begin{equation}\label{UB:ellpnu}
|\ell(p^{\nu})|\le C^{\nu}
\end{equation}
for all primes $p$ and all integers $\nu\ge 2$.
In fact, the definition of $\ell(n)$ allows us to write
\begin{align*}
0 
& = \lambda_f(p^{\nu})\mu(p^{\nu})^2 =\sum_{\nu_1+\nu_2+\nu_2=\nu} g(p^{\nu_1}) h(p^{\nu_2}) \ell(p^{\nu_3})
\\
& = \ell(p^{\nu}) + \ell(p^{\nu-1})(g(p)+h(p)) + \ell(p^{\nu-2})g(p)h(p)
\end{align*}
for all primes $p$ and all integers $\nu\ge 2$.
This implies
\begin{equation}\label{reccurence}
|\ell(p^{\nu})|\le (|g(p)|+|h(p)|+|g(p)h(p)|)\max\{|\ell(p^{\nu-1})|, |\ell(p^{\nu-2})|\}.
\end{equation}
In view of the definition of $g(n)$ and of $h(n)$ and \cite[Lemma 6]{LauRoyerWu2015}, 
we easily see that there is an absolute constant $C$ such that
\begin{equation}\label{initiale}
|g(p)|+|h(p)|+|g(p)h(p)|\le C.
\end{equation}
Clearly \eqref{UB:ellpnu} follows immediately from \eqref{reccurence} and \eqref{initiale}
by a simple recurrence. 
Now by using \eqref{def:elln}, \eqref{UB:ellpnu} and $\ell(p)=0$, 
the formula \eqref{Series_Convolution_2} holds for $\sigma>\tfrac{1}{2}$ with
$$
\mathscr{L}_f(s)
:= \sum_{\substack{n\ge 1\\ (n, N_k)=1}} \ell(n) n^{-s}
= \prod_{p\nmid N_k} \Big(1 + \sum_{\nu\ge 2} \ell(p^{\nu}) p^{-\nu s}\Big)
$$
and we have $\mathscr{L}_f(s)\asymp_{\varepsilon} 1$ for $\sigma\ge \tfrac{1}{2}+\varepsilon$
thanks to the fact that $p\nmid N_k$ implies $p\ge \mathcal{L}^2\ge (\log C_0)^2\ge 100C^2$,
since we have supposed that $C_0$ is a large constant.
Here the implied constants in the $\asymp_{\varepsilon}$-symbol depend only on $\varepsilon$.

Next we prove the assertion (ii).
Since $\chi_{t, N}^2=\chi_0$, Theorem II.8.7 of \cite{Tenenbaum1995} with $\chi=\chi_{t, N}$ gives us 
$$
L(\sigma, \chi_0)^3 |L(\sigma+\text{i}\tau, \chi_{t, N})|^4 |L(\sigma+\text{i}2\tau, \chi_0)|\ge 1
$$
for $\sigma>1$ and $\tau\in \R$.
On the other hand, 
for $\sigma\ge 1+\varepsilon$ and $\tau\in \R$ we have trivially 
$$
|L(\sigma, \chi_0)|+|L(\sigma+\text{i}2\tau, \chi_0)|
\le 2\zeta(1+\varepsilon)\ll \varepsilon^{-1},
$$
where the implied constant is absloute.
The inequality \eqref{LowerBound_Dirichlet_1} follows immediately.

For $\sigma>1$, we have
$$
L^{\flat}(s, \chi_{t, N})
= \prod_{p\nmid N_k} \bigg(1+\frac{\chi_{t, N}(p)}{p^s}\bigg)
= L(s, \chi_{t, N}) G_{\chi_{t, N}}(s),
$$
where the Dirichlet series of
$$
G_{\chi_{t, N}}(s)
:= \prod_{p\mid N_k} \bigg(1-\frac{\chi_{t, N}(p)}{p^s}\bigg)
\prod_{p\nmid N_k}
\bigg(1-\frac{\chi_{t, N}(p)^2}{p^{2s}}\bigg)
$$
converges absolutely and 
so $G_{\chi_{t, N}}(s)\gg_\varepsilon N^{-\varepsilon}$ in the half-plane $\sigma\ge \tfrac{1}{2}+\varepsilon$ 
(with the implied constant depending only on $\varepsilon$)
and $G_{\chi_{t, N}}(s)\gg \varepsilon$ for $\sigma\ge 1+\varepsilon$
(here the implied constant is absolute).
Now the inequality \eqref{LowerBound_Dirichlet_2} follows immediately from \eqref{LowerBound_Dirichlet_1}.

Finally we treat the assertion (iii).
Under our hypothesis, $f$ is a newform of weight $2k$ for the group $\Gamma_0(N/2)$ with character $\chi^2$.
Thus we have the convexity bound
\begin{equation}\label{ConvexityBound_Lsf}
L(s, f)\ll_{\varepsilon} \big(N(k+|\tau|)^2\big)^{\max\{(1-\sigma)/2, \, 0\}+\varepsilon}
\end{equation}
for $\sigma\ge \tfrac{1}{2}+\varepsilon$ and $\tau\in \R$
(see \cite[page 202, (1.22)]{Michel2007} or \cite[page 4, (1.12)]{Hacos2003}),
where the implied constant depends only on $\varepsilon$.

For $\sigma>1$, we have
$$
L^{\flat}(s, f)
= \prod_{p\nmid N_k} \bigg(1+\frac{\lambda_f(p)}{p^s}\bigg)
= L(s, f) G_f(s),
$$
where the Dirichlet series of
$$
G_f(s)
:= \prod_{p\mid N_k} \bigg(1-\frac{\lambda_f(p)}{p^s}\bigg)
\prod_{p\nmid N_k}
\bigg(1-\frac{\lambda_f(p)^2}{p^{2s}}\bigg)
$$
converges absolutely and so $G_f(s)\ll_\varepsilon (kN)^\varepsilon$ in the half-plane $\sigma\ge \tfrac{1}{2}+\varepsilon$ 
(as $|\lambda_f(p)|\le 2$ by Deligne's inequality \cite{Deligne1974}).
Using the convexity bound \eqref{ConvexityBound_Lsf}, we can derive
\begin{equation}\label{ConvexityBound_Lsf_flat}
L^{\flat}(s, f)
\ll_\varepsilon \big(N (k+|\tau|)^2\big)^{\max\{(1-\sigma)/2, \, 0\}+\varepsilon}
\end{equation}
for $\sigma\ge \tfrac{1}{2}+\varepsilon$ and $\tau\in \R$.
Here the implied constants in the $\ll_{\varepsilon}$-symbol depend only on $\varepsilon$.

By \eqref{Series_Convolution_1}, \eqref{LowerBound_Dirichlet_1} and \eqref{ConvexityBound_Lsf},
we get \eqref{ConvexityBounds_1}.
Similarly we can derive \eqref{ConvexityBounds_2} from \eqref{Series_Convolution_2}, \eqref{Estimate:mathscrLfs},
\eqref{LowerBound_Dirichlet_2} and \eqref{ConvexityBound_Lsf_flat}.
\end{proof}

The second lemma will be needed in the proof of Theorem \ref{thm2}.

\begin{lemma}\label{lem2.2}
Let $k\ge 1$ be an integer, $N\ge 4$ an integer divisible by $4$
and $\chi$ be a real Dirichlet character modulo $N$.
Suppose that ${\mathfrak{f}}\in \mathcal{S}_{k+1/2}^*(N, \chi)$ is a Hecke eigenform and 
$t\ge 1$ is a square-free integer such that ${\mathfrak a}_{\mathfrak{f}}(t)>0$.
Assume that its Shimura lift $f_t$ is not of CM type.
Then for any $\varepsilon>0$, there is a constant $x_0({\mathfrak{f}}, \varepsilon)$ such that
$$
\sum_{\substack{p\le x\\ {\mathfrak a}_{\mathfrak{f}}(tp^2)<0}} 1
\ge \bigg(\frac{1}{2}-\varepsilon\bigg)\frac{x}{\log x}
$$
for $x\ge x_0({\mathfrak{f}}, \varepsilon)$.
\end{lemma}

\begin{proof}
Taking $n=p$ in \eqref{Coefficient_Convolution}, it follows that 
$$
\mathfrak{a}_{\mathfrak{f}}^*(p)
= \lambda_f(p) - \frac{\chi_{t,N}(p)}{\sqrt{p}}\cdot
$$
In view of the hypothesis that ${\mathfrak a}_{\mathfrak{f}}(t)>0$ and \eqref{deflambdaf_a*f}, we have
$$
p>\varepsilon^{-2}\;\,\text{and}\;\,
\lambda_f(p)\le -\varepsilon
\;\Rightarrow\;
\mathfrak{a}_{\mathfrak{f}}^*(p)<0
\;\Rightarrow\;
\mathfrak{a}_{\mathfrak{f}}(tp^2)<0.
$$
Thus
$$
\sum_{\substack{p\le x\\ {\mathfrak a}_{\mathfrak{f}}(tp^2)<0}} 1
\ge \sum_{\substack{\varepsilon^{-2}<p\le x\\ \lambda_f(p)\le -\varepsilon}} 1.
$$
Now the required inequality is an immediate consequence of the Sato-Tate conjecture
(proved by Barnet-Lamb, Geraghty, Harris and Taylor \cite{BGHT2011}).
\end{proof}

The next lemma comes from the first part of Theorem 15 in Serre \cite{Serre1981}, 
which is the key tool for the proof of Theorem \ref{thm3}.

\begin{lemma}\label{lem2.3}
Let $g$ be any normalized Hecke eigenform of integral weight $\ge 2$ and of level $M$. 
Suppose that $\ell(X)\in \C[X]$ is any polynomial.
Write $a_g(n)$ for the $n$-th Fourier coefficient of $g$. 
If $g$ is not of CM type, then
\begin{equation}\label{SerreNewform}
\sum_{\substack{p\le x\\ p\nmid M, \, a_g(p)=\ell(p)}} 1
\ll_{g, \ell, \varepsilon} \frac{x}{(\log x)^{5/4-\varepsilon}}
\end{equation}
holds for any $\varepsilon>0$ and all $x\ge 2$,
where the implied constant depends on $g, \ell$ and $\varepsilon$.
\end{lemma}

\vskip 8mm

\section{The proof of Theorem \ref{thm1}}

Let $N_k$ be defined as in \eqref{defLNk}.
Consider the summatory function
\begin{equation}\label{defSfx}
S_{\mathfrak{f}}(x) 
:= \sum_{\substack{n\le x\\ (n, N_k)=1}} \mathfrak{a}_{\mathfrak{f}}^*(n) \mu(n)^2 \log\left(\frac{x}{n}\right)
= \mathop{{\sum}^{\flat}}_{\hskip -1,5mm n\le x} \mathfrak{a}_{\mathfrak{f}}^*(n) \log\left(\frac{x}{n}\right).
\end{equation}

\subsection{Upper bound for $S_{\mathfrak{f}}(x)$}\

\begin{proposition}\label{prop1}
Under the condition of Theorem \ref{thm1}, we have
\begin{equation}\label{UB_Sfx}
S_{\mathfrak{f}}(x)\ll_{\varepsilon} (k^2N)^{1/4+\varepsilon} x^{1/2}
\end{equation}
for all $x\ge 2$, where the implied constant depends on $\varepsilon$ only.
\end{proposition}

\begin{proof}
The Perron formula (cf. \cite[Theorem II.2.3]{Tenenbaum1995}) gives
$$
S_{\mathfrak{f}}(x) 
= \frac{1}{2\pi\i} \int_{2-\i\infty}^{2+\i\infty} L^{\flat}(s, \mathfrak{a}_{\mathfrak{f}}^*) \frac{x^{s}}{s^2} \d s.
$$
Moving the line of integration $\re s=2$ to $\re s=\dm+\varepsilon$ 
and applying the convexity bound \eqref{ConvexityBounds_2} for $L^{\flat}(s, \mathfrak{a}_{\mathfrak{f}}^*)$,
we obtain the required upper bound \eqref{UB_Sfx}.
\end{proof}

\subsection{Two preliminary lemmas}\

\vskip 1mm

In order to establish the required lower bound for $S_{\mathfrak{f}}(x)$,
we need two mean value theorems of multiplicative functions over friable (i.e. smooth) integers coprime with $q$.
For $x\ge 1$, $y\ge 2$ and $q\in \N$, define
$$
\Xi_{q}(x, y)
:= \sum_{\substack{n\le x, \, (n, q)=1\\ P(n)\le y}} \mu(n)^2 
\qquad{\rm and}\qquad
\Xi_{q}(x)
:= \Xi_{q}(x, x),
$$
where $P(n)$ is the largest prime factor of $n$.

The first lemma is a particular case of \cite[Lemma 4.2]{LauLiuWu2012} with $\kappa=1$.

\begin{lemma}\label{squrefreefriablecomprime}
Let $U>1$ be a fixed constant. For some suitable constant
$C=C(U)$ depending only on $U$, we have
\begin{equation}\label{Xiqyuy}
\begin{aligned}
\Xi_{q}(y^u, y)
& = \Pi_{q} \, y^u \rho(u)
\bigg\{1+O_{U}\bigg(\frac{L_q^{{\rm e}+2}}{\sqrt{\log y}}\bigg)\bigg\}
\end{aligned}
\end{equation}
uniformly for
\begin{equation}\label{cond.yq}
q\ge 1,
\qquad
y\ge \exp(2CL_q^{{\rm e}+2}),
\qquad
U^{-1}\le u\le U,
\end{equation}
where $L_q := \log(\omega(q)+3)$, 
$\omega(q)$ is the number of distinct prime factors of $q$,
$\varphi(q)$ is the Euler totient function and
\begin{equation}\label{defPiq}
\Pi_{q}
:= \frac{\varphi(q)}{q}
\prod_{p\nmid q} \bigg(1-\frac{1}{p^2}\bigg).
\end{equation}
Here $\rho(t)$ be the unique continuous solution of the difference-differential equation
\begin{equation}\label{defrho2}
\rho(t)=1 \quad (0\le t\le 1),
\qquad
t\rho'(t) = - \rho(t-1)
\quad
(t>1).
\end{equation}
The implied constant in the $O_U$-symbol depends only on $U$.
\end{lemma}

\vskip 1,5mm

We now introduce an auxiliary multiplicative function $h=h_{N_k, y}$ defined by
\begin{equation}\label{defhy1}
h_{N_k, y}(p)
:=\begin{cases}
0                                 & \text{if $\,p\mid N_k$},
\\\noalign{\vskip 0,5mm}
-2-c\mathcal{L}^{-1}   & \text{if $\,p>y$ and $p\nmid (N/2)$},
\\\noalign{\vskip 0,5mm}
-c\mathcal{L}^{-1}      & \text{if $\,\sqrt{y}<p\le y$ and $\,p\nmid (N/2)$},
\\\noalign{\vskip 0,5mm}
1-c\mathcal{L}^{-1}    & \text{if $\,\mathcal{L}^2\le p\le \sqrt{y}\,$ and $\,p\nmid (N/2)$},
\end{cases}
\end{equation}
and
\begin{equation}\label{defhy2}
h_{N_k, y}(p^\nu) := 0
\qquad
(\nu\ge 2),
\end{equation}
where $\mathcal{L}$ is defined as in \eqref{defLNk} and the constant $c>0$ will be chosen later.

\vskip 1mm

The next lemma is the key for giving a lower bound of $S_{\mathfrak{f}}(x)$.

\begin{lemma}\label{lm-h}
Let $k\ge 1$ be an integer and $N\ge 4$ an integer divisible by $4$.
Then for any $\varepsilon>0$, we have, for $N_k\to\infty$,
\begin{equation}\label{hyn0}
\sum_{n\le y^u } h_{N_k, y}(n) \log\left(\frac{y^u}{n}\right)
= \Pi_{N_k} y^u \big(\rho(2u)-2\log u\big)
\bigg\{1+O\bigg(\frac{1}{\log\mathcal{L}}\bigg)\bigg\}
\end{equation} 
uniformly for
\begin{equation}\label{ConditonuyN}
1\le u\le \tfrac{3}{2}
\qquad\text{and}\qquad
(k^2N)^{1/100}\le y\le (k^2N)^2,
\end{equation}
where $\Pi_{N_k}$ and $\rho(u)$ are defined as in Lemma \ref{squrefreefriablecomprime},
and the implied constant in the $O$-symbol is absolute.
In particular $\rho(2u)-2\log u>0$ for all $u < \kappa$ where $\kappa$ is
the solution to $\rho(2\kappa)=2\log \kappa$. 
We have $\kappa>\tfrac{10}{9}$.
\end{lemma}

\begin{proof}
According to the definition of $h_{N_k, y}$, we have
\begin{equation}\label{hyn1}
\sum_{n\le y^u} h_{N_k, y}(n) \log\left(\frac{y^u}{n}\right)
= S_1 + O(\mathcal{L}^{-1} S_2) - (2+c\mathcal{L}^{-1}) S_3,
\end{equation}
for all $u$ and $y$ satisfying \eqref{ConditonuyN},
where
\begin{align*}
S_1
& := \sum_{\substack{n\le y^u\\ P(n)\le \sqrt{y}}} h_{N_k, y}(n) \log\left(\frac{y^u}{n}\right), 
\\
S_2 
& := \sum_{\substack{\sqrt{y}<p\le y\\ p\nmid N_k}} \sum_{\substack{n\le y^u/p\\ p\nmid n}} 
h_{N_k, y}(n) \log\left(\frac{y^u}{pn}\right),
\\
S_3 
& := \sum_{\substack{y<p\le y^u\\ p\nmid N_k}} \sum_{n\le y^u/p} 
h_{N_k, y}(n) \log\left(\frac{y^u}{pn}\right).
\end{align*}

For square-free $n\le y^u\le (k^2N)^3$ with $P(n)\le \sqrt{y}$ and $(n, N_k)=1$, we have
\begin{align*}
h_{N_k, y}(n)
& = \bigg(1-\frac{c}{\mathcal{L}}\bigg)^{\omega(n)}
= \exp\bigg\{\omega(n)\log\bigg(1-\frac{c}{\mathcal{L}}\bigg)\bigg\}
\\
& = \exp\bigg\{O\bigg(\frac{1}{\log\mathcal{L}}\bigg)\bigg\}
= 1 + O\bigg(\frac{1}{\log\mathcal{L}}\bigg),
\end{align*}
where the implied constants are absolute.
Thus
\begin{align*}
S_1
& = \bigg\{1 + O\bigg(\frac{1}{\log\mathcal{L}}\bigg)\bigg\} 
\sum_{\substack{n\le y^u, \, (n, N_k)=1, \, P(n)\le \sqrt{y}}} \mu(n)^2 \log\bigg(\frac{y^u}{n}\bigg)
\\
& = \bigg\{1 + O\bigg(\frac{1}{\log\mathcal{L}}\bigg)\bigg\} 
\int_{1-}^{y^u} \log\bigg(\frac{y^u}{t}\bigg) \d \Xi_{N_k}(t, \sqrt{y})
\\\noalign{\vskip 0,5mm}
& = \bigg\{1 + O\bigg(\frac{1}{\log\mathcal{L}}\bigg)\bigg\} 
\int_{1}^{y^u} \frac{\Xi_{N_k}(t, \sqrt{y})}{t} \d t.
\end{align*}
Here the implied constants are absolute.
By Lemma \ref{squrefreefriablecomprime} with $(q, y, y^u) = (N_k, \sqrt{y}, t)$, it follows that
$$
\int_{1}^{y^u} \frac{\Xi_{N_k}(t, \sqrt{y})}{t} \d t
= \bigg\{1+O\bigg(\frac{\log^{\text{e}+2}\!\mathcal{L}}{\sqrt{\mathcal{L}}}\bigg)\bigg\}
\Pi_{N_k} \int_{1}^{y^u} \rho\bigg(\frac{\log t}{\log\sqrt{y}}\bigg) \d t,
$$
where the implied constant is absolute.
By making the change of variables $t=y^{u-v/2}$ and by partial integration, we deduce
\begin{align*}
\int_{1}^{y^u} \rho\bigg(\frac{\log t}{\log\sqrt{y}}\bigg) \d t
& = y^u (\log\sqrt{y}) \int_{0}^{2u} y^{-v/2} \rho(2u-v) \d v
\\
& = y^u \bigg\{\rho(2u) - 
\bigg(\int_{0}^{1/\sqrt{\mathcal{L}}} + \int_{1/\sqrt{\mathcal{L}}}^{2u}\bigg) y^{-v/2} \rho'(2u-v) \d v\bigg\}
\\
& = y^u \rho(2u) \bigg\{1+O\bigg(\frac{1}{\sqrt{\mathcal{L}}}\bigg)\bigg\},
\end{align*}
where the implied constant is absolute.
Combining these estimations, we find that
\begin{equation}\label{S1}
S_1
= \bigg\{1+O\bigg(\frac{1}{\log\mathcal{L}}\bigg)\bigg\}
y^u \Pi_{N_k} \rho(2u),
\end{equation}
where the implied constant is absolute.

Similarly we can write
$$
S_3 = S_3' - S_3'',
$$
where
\begin{align*}
S_3' 
& = \bigg\{1+O\bigg(\frac{1}{\log\mathcal{L}}\bigg)\bigg\}
\sum_{y<p\le y^u} \sum_{\substack{n\le y^u/p\\ (n, N_k)=1}} \mu(n)^2 \log\bigg(\frac{y^u}{pn}\bigg)
\\
& = \bigg\{1+O\bigg(\frac{1}{\log\mathcal{L}}\bigg)\bigg\}
\sum_{y<p\le y^u} 
\int_{1-}^{y^u/p} \log\bigg(\frac{y^u}{pt}\bigg) \d \Xi_{N_k}(t) 
\\
& = \bigg\{1+O\bigg(\frac{1}{\log\mathcal{L}}\bigg)\bigg\}
\sum_{y<p\le y^u} 
\int_{1}^{y^u/p} \frac{\Xi_{N_k}(t)}{t} \d t
\\
& = \bigg\{1+O\bigg(\frac{1}{\log\mathcal{L}}\bigg)\bigg\}
y^u \Pi_{N_k} \sum_{y<p\le y^u} \frac{1}{p}
\\
& = \bigg\{1+O\bigg(\frac{1}{\log\mathcal{L}}\bigg)\bigg\} y^u \Pi_{N_k} \log u,
\end{align*}
and
\begin{align*}
S_3''
& \ll \sum_{\substack{y<p\le y^u\\ p\mid N_k}} 
\sum_{\substack{n\le y^u/p\\ (n, N_k)=1}} \mu(n)^2 \log\bigg(\frac{y^u}{pn}\bigg)
\ll y^u \Pi_{N_k} \sum_{\substack{y<p\le y^u\\ p\mid N_k}} \frac{1}{p}
\\
& \ll y^{u-1} \Pi_{N_k} \omega(N_k)
\ll y^{u} \Pi_{N_k} (\log\mathcal{L})^{-1},
\end{align*}
where we have used the following estimates :
$\omega(N_k)\ll \log N_k\ll \mathcal{L}^2\ll y (\log\mathcal{L})^{-1}$ since $y\ge (k^2N)^{1/100}$.
Here the implied constants are absolute.
These imply that
\begin{equation}\label{S3}
S_3 = \bigg\{1+O\bigg(\frac{1}{\log\mathcal{L}}\bigg)\bigg\} y^u \Pi_{N_k} \log u,
\end{equation}
where the implied constant is absolute.

Finally we have
\begin{equation}\label{S2}
\begin{aligned}
S_2 
& \le \sum_{\sqrt{y}<p\le y} \sum_{n\le y^u/p, \, (n, N_k)=1} \mu(n)^2 \log\bigg(\frac{y^u}{pn}\bigg)
\\
& = \sum_{\sqrt{y}<p\le y} \int_{1-}^{y^u/p} \log\bigg(\frac{y^u}{pt}\bigg) \d \Xi_{N_k}(t)
\\
& = \sum_{\sqrt{y}<p\le y} \int_1^{y^u/p} \frac{\Xi_{N_k}(t)}{t} \d t
\\
& \ll y^u \Pi_{N_k} \sum_{\sqrt{y}<p\le y} \frac{1}{p}
\\
& \ll y^u \Pi_{N_k},
\end{aligned}
\end{equation}
where the implied constants are absolute.

Inserting \eqref{S1}, \eqref{S2} and \eqref{S3} into \eqref{hyn1}, we get \eqref{hyn0}.
\end{proof}

\subsection{Lower bound for $S_{\mathfrak{f}}(x)$}\

\vskip 1mm

From \eqref{Coefficient_Convolution}, the inversion formula of Möbius allows us to deduce
\begin{equation}\label{fFpnu}
{\mathfrak a}_{\mathfrak{f}}^*(n)
= \sum_{d\mid n} \frac{\mu(d) \chi_{t,N}(d)}{\sqrt{d}} \lambda_f\bigg(\frac{n}{d}\bigg).
\end{equation} 
Taking $n=p^{\nu}$, it follows that 
\begin{equation}\label{fFpnu_2}
{\mathfrak a}_{\mathfrak{f}}^*(p^{\nu})
= \lambda_f(p^{\nu}) - \frac{\chi_{t,N}(p)}{\sqrt{p}} \lambda_f(p^{\nu-1}).
\end{equation} 
Thus
\begin{equation}\label{fFpnu_3}
{\mathfrak a}_{\mathfrak{f}}^*(p^{\nu})\ge 0
\; \Leftrightarrow \;
\lambda_f(p^{\nu})\ge \frac{\chi_{t,N}(p)}{\sqrt{p}} \lambda_f(p^{\nu-1})
\; \Rightarrow \;
\lambda_f(p^{\nu})\ge -\frac{\nu}{\sqrt{p}},
\end{equation} 
where we have used the Deligne bound.

Since $f$ is a newform of weight $2k$ for the group $\Gamma_0(N/2)$ with character $\chi^2$, 
for each prime $p\nmid (N/2)$ there is a unique  real $\theta_f(p)\in [0, \pi]$ such that
\begin{equation}\label{Chebyshev}
\lambda_f(p^{\nu}) = \frac{\sin((\nu+1)\theta_f(p))}{\sin\theta_f(p)}
\quad
(\nu\ge 1).
\end{equation} 

\vskip 1mm

Let $y_{\mathfrak{f}}>0$ be the largest integer such that
\begin{equation}\label{HPositivity}
{\mathfrak a}_{\mathfrak{f}}^*(n)\ge 0
\qquad\text{for}\quad
n\le y_{\mathfrak{f}}
\quad\text{and}\quad
(n, N/2)=1.
\end{equation}
We now proceed to establish a lower bound for $S_{\mathfrak{f}}(x)$ 
by using the assumption of positivity \eqref{HPositivity}. 
For primes $\mathcal{L}^2\le p\le y_{\mathfrak{f}}$ with $p\nmid (N/2)$, 
we thus have ${\mathfrak a}_{\mathfrak{f}}(tp^2)\geq 0$.
From \eqref{fFpnu_3} and \eqref{Chebyshev} with $\nu=1$, it follows that
$$
\lambda_f(p) = \frac{\sin(2\theta_f(p))}{\sin\theta_f(p)}
\ge -\frac{1}{\sqrt{p}}
\ge -\frac{2}{\mathcal{L}}\cdot
$$
Hence, there is an absolute positive constant $c_1>0$ 
such that $\theta_f(p)\le \tfrac{\pi}{2}+\tfrac{c_1}{\mathcal{L}}$.
Furthermore, if $\mathcal{L}^2\le p\le \sqrt{y_{\mathfrak{f}}}$ and $p\nmid (N/2)$, 
we have
$$
\lambda_f(p^2) = \frac{\sin(3\theta_f(p))}{\sin\theta_f(p)}
\ge -\frac{2}{\sqrt{p}}
\ge -\frac{2}{\mathcal{L}}\cdot
$$
This implies that $\theta_f(p)\le \tfrac{\pi}{3}+\tfrac{c'}{\mathcal{L}}$ and
$$
\lambda_f(p)\ge 2\cos\bigg(\frac{\pi}{3}+\frac{c'}{\mathcal{L}}\bigg)
\ge 1-\frac{c}{2\mathcal{L}},
$$
where $c>0$ is an absolute positive constant.

In view of \eqref{fFpnu_2} with $\nu=1$ and the definition of $h_{N_k, y}(p)$, it is clear that
\begin{equation}\label{af_hNy}
\mathfrak{a}_{\mathfrak{f}}^*(p)\ge h_{N_k, y_{\mathfrak{f}}}(p)
\end{equation}
for all prime numbers $p\nmid N_k$.

\begin{proposition}\label{prop2}
Let $k\ge 1$ be an integer, $N\ge 4$ an integer divisible by $4$
and $\chi$ be a real Dirichlet character modulo $N$.
Suppose that ${\mathfrak{f}}\in \mathcal{S}_{k+1/2}^*(N, \chi)$ is a Hecke eigenform and 
$t\ge 1$ is a square-free integer such that ${\mathfrak a}_{\mathfrak{f}}(t)>0$.
Assume that its Shimura lift is not of CM type
and that $(k^2N)^{1/100}\le y_{\mathfrak{f}}\le (k^2N)^{2}$. Then we have
\begin{equation}\label{LB_Sfx}
S_{\mathfrak{f}}(y_{\mathfrak{f}}^u)\gg y_{\mathfrak{f}}^u (\log_2(kN))^{-1}
\end{equation}
for all $u<\kappa$,
where $y_{\mathfrak{f}}$ and $\kappa$ are defined as in \eqref{HPositivity} and  Lemma \ref{lm-h}, respectively,
and $\log_r$ means the $r$-fold iterated logarithm.
Here the implied constant is absolute.
\end{proposition}

\begin{proof}
With the help of \eqref{multiplicativity}, it is easy to verify that $n\mapsto \mathfrak{a}_{\mathfrak{f}}^*(n)$ is multiplicative.
Let $g_{N_k, y_{\mathfrak{f}}}$ be the multiplicative function defined by the
Dirichlet convolution identity $\mathfrak{a}_{\mathfrak{f}}^*\mu^2=g_{N_k, y_{\mathfrak{f}}}*h_{N_k, y_{\mathfrak{f}}}$. 
Then $g_{N_k, y_{\mathfrak{f}}}(n)\geq 0$ for all square-free integers $n\geq 1$ with $(n, \, N_k)=1$, since
$g_{N_k, y_{\mathfrak{f}}}(p)=\mathfrak{a}_{\mathfrak{f}}^*(p)-h_{N_k, y_{\mathfrak{f}}}(p)\geq 0$ for $p\nmid N_k$
thanks to \eqref{af_hNy}.

On the other hand, according to Lemma \ref{lm-h}, we have
$$
\sum_{n\leq y_{\mathfrak{f}}^u}{h_{N_k, y_{\mathfrak{f}}}(n)} \log\left(\frac{y_{\mathfrak{f}}^u}{n}\right)
\geq 0
$$
for $u<\kappa$ and sufficiently large $y_{\mathfrak{f}}$.
But, as $g_{N_k,y_{\mathfrak{f}}}(1)=1$, 
we infer that
\begin{align*}
S_{\mathfrak{f}}(y_{\mathfrak{f}}^u)
& = \sumb_{n\le y_{\mathfrak{f}}^u} g_{N_k, y_{\mathfrak{f}}}*h_{N_k, y_{\mathfrak{f}}}(n) \log\left(\frac{y_{\mathfrak{f}}^u}{n}\right)
\\
& = \sumb_{d\le y_{\mathfrak{f}}^u} g_{N_k, y_{\mathfrak{f}}}(d)
\sumb_{m\leq y_{\mathfrak{f}}^u/d} h_{N_k, y_{\mathfrak{f}}}(m) \log\left(\frac{y_{\mathfrak{f}}^u}{dm}\right)
\\
& \geq \sumb_{m\leq y_{\mathfrak{f}}^u} h_{N_k, y_{\mathfrak{f}}}(m) \log\left(\frac{y_{\mathfrak{f}}^u}{m}\right).
\end{align*}
Then we get the required lower bound \eqref{LB_Sfx} by Lemma \ref{lm-h} ,
since we have, by (\ref{defPiq}) and the prime number theorem,
$$
\Pi_{N_k}
\gg \frac{\varphi(N_k)}{N_k}
\gg (\log_2N_k)^{-1}
\gg \{\log_2(kN)\}^{-1},
$$
where the implied constant is absolute.
This completes the proof.
\end{proof}

\subsection{End of the proof of Theorem \ref{thm1}}\

\vskip 1mm

Without loss of generality, we can assume that $y_{\mathfrak{f}}\ge (k^2N)^{1/100}$.

Firstly we prove
\begin{equation}\label{UB:First}
y_{\mathfrak{f}}\ll (k^2N)^2,
\end{equation}
where the implied constant is absolute.

By the definition of $y_{\mathfrak{f}}$, the multiplicativity of $\mathfrak{a}_{\mathfrak{f}}^*(n)$ and \eqref{af_hNy}, we have
\begin{align*}
S_{\mathfrak{f}}(y_{\mathfrak{f}})
& \ge \sum_{\substack{\mathcal{L}^2\le p\not=p'\le y_{\mathfrak{f}}^{1/3}\\ (pp', N/2)=1}} 
\mathfrak{a}_{\mathfrak{f}}^*(pp') \log\bigg(\frac{y_{\mathfrak{f}}}{pp'}\bigg)
\\
& \ge 3(\log y_{\mathfrak{f}}) \sum_{\substack{\mathcal{L}^2\le p\not=p'\le y_{\mathfrak{f}}^{1/3}\\ (pp', N/2)=1}} 
\Big(1-\frac{c}{2\mathcal{L}}\Big)^{2}
\\
& \gg (\log y_{\mathfrak{f}})
\bigg\{\Big(\sum_{\substack{\mathcal{L}^2\le p\le y_{\mathfrak{f}}^{1/3}\\ (p, N/2)=1}} 1\Big)^2
- \sum_{p\le y_{\mathfrak{f}}^{1/3}} 1\bigg\}
\\
& \gg y_{\mathfrak{f}}^{2/3},
\end{align*}
where the implied constant is absolute.
Combining this with Proposition \ref{prop1} yields
$$
y_{\mathfrak{f}}^{2/3}
\ll S_{\mathfrak{f}}(y_{\mathfrak{f}})
\ll_{\varepsilon} (k^2N)^{1/4+\varepsilon} y_{\mathfrak{f}}^{1/2}
$$
where the first implied constant is absolute and the second depends only on $\varepsilon$.
Clearly these imply the required inequality \eqref{UB:First}.

In view of \eqref{UB:First}, 
we can apply Propositions \ref{prop1} and \ref{prop2} to write
\begin{equation}\label{proof:thm1_A}
y_{\mathfrak{f}}^u (\log\mathcal{L})^{-1}
\ll S_{\mathfrak{f}}(y_{\mathfrak{f}}^u)
\ll_{\varepsilon} (k^2N)^{1/4+\varepsilon} y_{\mathfrak{f}}^{u/2}
\end{equation}
for $u<\kappa$, 
where the first implied constant is absolute and the second depends only on $\varepsilon$.
From \eqref{proof:thm1_A} we deduce that $y_{\mathfrak{f}}\ll_{\varepsilon} (k^2N)^{1/(2u)+\varepsilon}$ for $1\le u<\kappa$.
According to Lemma \ref{lm-h}, we know $\kappa>\tfrac{10}{9}$.
Thus $y_{\mathfrak{f}}\ll (k^2N)^{9/20}$, where the implied constant is absolute.
This is equivalent to the result of Theorem \ref{thm1}.

\vskip 10mm

\section{Proof of Theorem \ref{thm2}}

\smallskip

Consider the function
$$
\varepsilon_{\mathfrak{f}}(n)
= \text{sign}\,{\mathfrak a}_{\mathfrak{f}}(tn^2)
:= \begin{cases}
1 & \text{si $\, {\mathfrak a}_{\mathfrak{f}}(tn^2)>0$},
\\
-1 & \text{si $\, {\mathfrak a}_{\mathfrak{f}}(tn^2)<0$},
\\
0 & \text{si $\, {\mathfrak a}_{\mathfrak{f}}(tn^2)=0$}.
\end{cases}
$$
By using \eqref{multiplicativity}, it is easy to check that this function is multiplicative.
According to \cite[Theorem]{HallTenenbaum1991},
for any real multiplicative function such that $|g(n)|\leq 1$, the inequality
$$
\sum_{n\leq x} g(n)\ll x \exp\bigg(-K\sum_{p\leq x} \frac{1-g(p)}{p}\bigg)
$$
holds for all $x\ge 2$,
where $K = 0.32867 \dots = -\cos\phi_0$ and 
$\phi_0$ is the unique root in $(0, \pi)$ of the equation $\sin\phi-\phi\cos\phi=\frac{1}{2}\pi$.
Applying this result to $\varepsilon_{\mathfrak{f}}(n)$, it follows that
\begin{align*}
\sum_{n\leq x} \varepsilon_{\mathfrak{f}}(n)
& \ll x \exp\bigg(-K\sum_{p\leq x} \frac{1-\varepsilon_{\mathfrak{f}}(p)}{p}\bigg)
\\
& \ll x \exp\bigg(-K\sum_{p\leq x, \, \varepsilon_{\mathfrak{f}}(p)=-1} \frac{2}{p}\bigg).
\end{align*}
With the help of Lemma \ref{lem2.2}, a simple integration by parts gives us
$$
\sum_{p\leq x, \, \varepsilon_{\mathfrak{f}}(p)=-1} \frac{1}{p}
= \int_{2-}^x \frac{1}{t} \d \Big(\sum_{p\leq t, \, \varepsilon_{\mathfrak{f}}(p)=-1} 1\Big)
\ge \bigg(\frac{1}{2}-\varepsilon\bigg) \log_2x
$$
for any $\varepsilon>0$ and all $x\ge x_0({\mathfrak{f}}, \varepsilon)$.

Combining these two estimates, we find that 
\begin{equation}\label{thm1.EqA}
\sum_{n\leq x} \varepsilon_{\mathfrak{f}}(n)
\ll_{{\mathfrak{f}}, \varepsilon} \frac{x}{(\log x)^{K-\varepsilon}}
\end{equation}
for any $\varepsilon>0$ and all $x\ge 2$.

On the other hand, a particular case of \cite[Theorem 2]{LiuWu2015} can be stated as follows:
Let $h$ be a non-negative multiplicative function satisfying the following conditions 
\begin{align}
& \sum_{p\leq z} h(p)\log p = \kappa z + O\bigg(\frac{z}{(\log z)^{\delta}}\bigg) \quad (z\geq 2),
\label{TWcondition1}
\\
& \sum_{p, \, \nu\geq 2} \frac{h(p^\nu)}{p^\nu} \log p^\nu\leq A,
\label{TWcondition2}
\end{align}
where $A>0$, $\kappa>0$ and $\delta>0$ are constants.
Then we have
\begin{equation}\label{AsymptoticSfx}
\sum_{n\leq x} h(n) = C_h x(\log x)^{\kappa-1} \bigg\{1
+ O_{h, \delta}\bigg(\frac{\log_2x}{\log x} + \frac{1}{(\log x)^{\delta}}\bigg)\bigg\},
\end{equation}
where
$$
C_h := \prod_p \big(1-p^{-1}\big)^\kappa 
\sum_{\nu\geq 0} g(p^{\nu}) p^{-\nu}.
$$
In view of Lemma \ref{lem2.3} and the prime number theorem, we immediately see that
$$
\sum_{p\leq z} |\varepsilon_{\mathfrak{f}}(p)|\log p 
= z + O_{\mathfrak{f}, \varepsilon}\big(z (\log z)^{-1/4+\varepsilon}\big) \quad (z\geq 2).
$$
This shows that the non-negative multiplicative function $n\mapsto |\varepsilon_{\mathfrak{f}}(n)|$ 
satisfies the condition \eqref{TWcondition1}.
Since $|\varepsilon_{\mathfrak{f}}(n)|\le 1$ for all $n\ge 1$,
the condition \eqref{TWcondition2} is verified trivially.
Thus \eqref{AsymptoticSfx} implies that 
\begin{equation}\label{thm1.EqB}
\sum_{n\leq x} |\varepsilon_{\mathfrak{f}}(n)| 
= \rho_{\mathfrak{f}} \, x 
\big\{1 + O_{{\mathfrak{f}}, \varepsilon}\big((\log x)^{-1/4+\varepsilon}\big)\big\}.
\end{equation}
This is equivalent to the assertion (i).

Notice that
$$
\frac{|\varepsilon_{\mathfrak{f}}(n)| + \varepsilon_{\mathfrak{f}}(n)}{2}
= \begin{cases}
1 & \text{if $\varepsilon_{\mathfrak{f}}(n)=1$},
\\\noalign{\vskip 1mm}
0 & \text{otherwise}.
\end{cases}
$$
The estimates \eqref{thm1.EqA} and \eqref{thm1.EqB} imply that
\begin{equation}\label{thm1.EqC}
\sum_{\substack{n\leq x\\ \varepsilon_{\mathfrak{f}}(n)=1}} 1 
= \tfrac{1}{2} \rho_{\mathfrak{f}} \, x 
\big\{1 + O_{{\mathfrak{f}}, \varepsilon}\big((\log x)^{-1/4+\varepsilon}\big)\big\}.
\end{equation}
This is equivalent to the assertion (ii) with sign +.
The other case can be treated in the same way.

\vskip 8mm

\section{Proof of Theorem \ref{thm3}}

In order to generalize the square-free numbers, 
Erd\H os \cite{Erdos1966} introduced the notion of $\mathscr{B}$-free numbers. 
More precisely, let 
$$
\mathscr{B}=\{b_i \,:\, 1<b_1<b_2<\dots\,\}
$$
be a sequence of integers verifying the following conditions
\begin{equation}\label{Bhypothese1}
\sum_{i\ge 1} b_i^{-1}<\infty
\end{equation}
and
\begin{equation}\label{Bhypothese2}
(b_i,b_j)=1
\quad
(i\not=j).
\end{equation}
A positive integer $n\geq 1$ is called $\mathscr{B}$-free
if it is not divisible
by any element in $\mathscr{B}$.
Many authors studied the distribution of $\mathscr{B}$-free integers.
A detailed historical description can be found in \cite{KowalskiRobertWu2007, WuZhai2013}.
In particular, by using sieve and estimates for multiple exponential sums, 
the authors of these two papers proved the following results 
(see \cite[Corollary 10]{KowalskiRobertWu2007} and \cite[Proposition 2]{WuZhai2013}, respectively):
\par
$\bullet$
\textsl{For all $\varepsilon>0$, $x\ge x_0(\mathscr{B}, \varepsilon)$ and $y\ge x^{7/17+\varepsilon}$,
we have
\begin{equation}\label{krw}
|\{x<n\le x+y\,:\,
n \;\hbox{is}\;\mathscr{B}\hbox{-free}\}|
\gg_{\mathscr{B},\varepsilon} y.
\end{equation}}

$\bullet$
\textsl{For all $\varepsilon>0$,
$x\ge x_0(\mathscr{B}, \varepsilon)$,
$y\ge x^{17/38+100\varepsilon}$,
$1\le a\le q\le x^\varepsilon$ with $((a, q), b)=1$
for all $b\in \mathscr{B}$, we have
\begin{equation}\label{wz}
|\{x<n\le x+y\,:\,
n\equiv a\,({\rm mod}\, q)
\;\hbox{and}\; n \;\hbox{is}\;\mathscr{B}\hbox{-free}\}|
\gg_{\mathscr{B},\varepsilon} y/q.
\end{equation}}
Here the implied constants depend only on $\mathscr{B}$ and $\varepsilon$.

Now let $\mathscr{P}$ be the set of all primes and define
$$
\mathfrak{P}_{\mathfrak{f}} 
:= \{p\in \mathscr{P} \,:\, {\mathfrak a}_{\mathfrak{f}}(tp^2)=0\},
\qquad
\mathscr{B}_{\mathfrak{f}}
:= \mathfrak{P}_{\mathfrak{f}}\cup \{p^2 \,:\, p\in \mathscr{P}\sset\mathfrak{P}_{\mathfrak{f}}\}.
$$
Clearly if $n$ is $\mathscr{B}_{\mathfrak{f}}$-free, 
then $n$ is square-free and by all its prime factors are not in $\mathfrak{P}_{\mathfrak{f}}$.
By using \eqref{multiplicativity}, it is easy to see that
${\mathfrak a}_{\mathfrak{f}}(tn^2)\not=0$ for all $\mathscr{B}_{\mathfrak{f}}$-free numbers $n$.
Thus \eqref{krw} and \eqref{wz} imply the first and second assertions of Theorem \ref{thm3}, respectively,
if we can show that the sequence $\mathscr{B}_{\mathfrak{f}}$ verifies the conditions 
\eqref{Bhypothese1} and \eqref{Bhypothese2}.

Firstly, the definition of $\mathscr{B}_{\mathfrak{f}}$ guarantees that 
the condition \eqref{Bhypothese2} is satisfied by $\mathscr{B}_{\mathfrak{f}}$.
On the other hand, in view of \eqref{defaft}, we have
$$
a_{f_t}(p)
= {\mathfrak a}_{\mathfrak{f}}\big(tp^2\big)
+ \chi_{t,N}(p) {\mathfrak a}_{\mathfrak{f}}(t) p^{k-1}.
$$
Thus \eqref{SerreNewform} of Lemma \ref{lem2.3} implies that
$$
\sum_{\substack{p\le x\\ {\mathfrak a}_{\mathfrak{f}}(tp^2)=0}} 1
= \sum_{\substack{p\le x\\ a_{f_t}(p)=\chi_{t,N}(p) {\mathfrak a}_{\mathfrak{f}}(t) p^{k-1}}} 1
\ll_{\mathfrak{f}} \frac{x}{(\log x)^{5/4-\varepsilon}},
$$
where the implied constant depends on $\mathfrak{f}$.
Hence a simple integration by parts gives us
\begin{align*}
\sum_{p\le x, \, p\in \mathfrak{P}_{\mathfrak{f}}} \frac{1}{p}
& = \int_{2-}^x \frac{1}{t} \d\Big(\sum_{p\le t, \, {\mathfrak a}_{\mathfrak{f}}(tp^2)=0} 1\Big)
\\\noalign{\vskip -1mm}
& = \frac{1}{x} \sum_{p\le x, \, {\mathfrak a}_{\mathfrak{f}}(tp^2)=0} 1
+ \int_{2}^x \frac{1}{t^2} \Big(\sum_{p\le t, \, {\mathfrak a}_{\mathfrak{f}}(tp^2)=0} 1\Big) \d t
\\\noalign{\vskip -1mm}
& \ll_{\mathfrak{f}} \frac{1}{(\log x)^{5/4-\varepsilon}} 
+ \int_{2}^x \frac{\d t}{t(\log t)^{5/4-\varepsilon}} 
\\\noalign{\vskip 2,5mm}
& \ll_{\mathfrak{f}} 1,
\end{align*}
where the implied constant depends on $\mathfrak{f}$.
This implies that the sequence $\mathscr{B}_{\mathfrak{f}}$ verifies the condition \eqref{Bhypothese1}.


\vskip 10mm

\begin{thebibliography}{CC}

\bibitem{BGHT2011}
T. Barnet-Lamb, D. Geraghty, M. Harris and R. Taylor,
\textit{A family of Calabi-Yau varieties and potential automorphy. II},
Pub. Res. Inst. Math. Sci., {\bf 47} (2011), no. 1, 29--98.

\bibitem{BK08}
J.-H. Bruinier and W. Kohnen,
\textit{Sign changes of coefficients of half integral weight modular forms},
In: Modular forms on Schiermonnikoong (eds. B. Edixhoven et. al.), 57--66, Cambridge Univ. Press, 2008.

\bibitem{Cipra1983}
B. A. Cipra, 
\textit{On the Niwa-Shintani theta-kernel lifting of modular forms}, Nagoya Math. J. {\bf 91} (1983), 49--117. 

\bibitem{Deligne1974}
P. Deligne,
\textit{La conjecture de Weil, I, II},
Publ. Math. IHES {\bf 48} (1974), 273--308;
{\bf 52} (1981), 313--428.

\bibitem{Erdos1966}
P. Erd\H os, 
\textit{On the difference of consecutive terms of sequences, defined by divisibility properties},
Acta Arith. {\bf 12} (1966), 175--182.

\bibitem{HallTenenbaum1991}
R.~R. Hall and G.~Tenenbaum, 
\textit{Effective mean value estimates for complex multiplicative functions}, 
Math. Proc. Cambridge Philos. Soc. {\bf 110} (1991), no. 2, 337--351.

\bibitem{Hacos2003}
G. Harcos,
\textit{New bounds for automorphic $L$-functions},
Ph. D. thesis, Princeton University, 2003.

\bibitem{Kohnen1985}
W. Kohnen, 
\textit{Fourier coefficients of modular forms of half integral weight}, 
Math. Ann. {\bf 271} (1985), 237--268.

\bibitem{KohnenLauWu2013}
W. Kohnen, Y.-K. Lau and J. Wu, 
\textit{Fourier coefficients of cusp forms of half-integral weight}, 
Math. Z. {\bf 273} (2013), 29--41.

\bibitem{KohnenZagier1981}
W. Kohnen and D. Zagier, 
\textit{Values of $L$-series of modular forms at the center of the critical strip}, 
Invent. Math. {\bf 64} (1981), 175--198.

\bibitem{KowalskiRobertWu2007}
E. Kowalski, O. Robert and J. Wu, 
\textit{Small gaps in coefficients of {$L$}-functions and $\mathscr{B}$-free numbers in short intervals},
Rev. Mat. Iberoam. {\bf 23} (2007), no.~1, 281--326. 

\bibitem{KLSW2010}
E. Kowalski, Y.-K. Lau, K. Soundararajan and J. Wu,
\textit{On modular signs}, 
Math. Proc. Camb. Phil. Soc. {\bf 149} (2010), 389--411.

\bibitem{LauLiuWu2012}
Y.-K. Lau, J.-Y. Liu and J. Wu,
\textit{The first negative coefficients of symmetric square $L$-functions}, 
Ramanujan J. {\bf 27} (2012), no. 3, 419--441.

\bibitem{LauRoyerWu2015}
Y.-K. Lau, E. Royer and J. Wu,
\textit{Sign of Fourier coefficients of modular forms of half integral weight},
arXiv:1507.00518v1 [math.NT] 2 Jul 2015.


\bibitem{LiuWu2015}
J.-Y. Liu and J. Wu, 
\textit{The number of coefficients of automorphic $L$-functions for $GL_m$ of same signs}, 
J. Number Theory {\bf 148} (2015), 429--450..

\bibitem{Matomaki2012}
K. Matom\"aki,
\textit{On signs of Fourier coeffcients of cusp forms},
Math. Proc. Camb. Phil. Soc. {\bf 152} (2012), 207--222.

\bibitem{MatomakiRadzwill2014}
K. Matom\"aki and M. Radzwill,
\textit{Sign changes of Hecke eigenvalues}, 
arXiv~:~1405.7671v1 [math.NT].

\bibitem{Michel2007}
P. Michel,
\textit{Analytic number theory and families of automorphic $L$-functions}, 
in: Automorphic forms and applications, 181--295, IAS/Park City
Mathematics Series, {\bf 12}, American Mathematical Society, Providence, RI, 2007.

\bibitem{Serre1981}
J.-P. Serre, 
\textit{Quelques applications du th\'eor\`eme de densit\'e de Chebotarev},
Inst. Hautes \'Etudes Sci. Publ. Math. {\bf 54} (1981), 323--401.

\bibitem{Shimura1973}
G. Shimura, 
\textit{On modular forms of half integral weight},
Ann. of Math. (2) {\bf 97} (1973), 440--481.

\bibitem{Tenenbaum1995}
G. Tenenbaum,
\textit{Introduction to analytic and probabilistic number theory},
Translated from the second French edition (1995) by C. B. Thomas,
Cambridge Studies in Advanced Mathematics {\bf 46},
Cambridge University Press, Cambridge, 1995. xvi+448 pp.

\bibitem{WuZhai2013}
J. Wu and W.-G. Zhai, 
\textit{Distribution of {H}ecke eigenvalues of newforms in short intervals}, 
Q. J. Math. \textbf{64} (2013), no.~2, 619--644.

\end{thebibliography}
\end{document}